\documentclass[11pt]{amsart}  
\usepackage[leqno]{amsmath}
\usepackage{amssymb, eucal,amsthm}
\usepackage[all,arc]{xy}
\usepackage{enumerate}
\usepackage{color}
\usepackage{stmaryrd}
\usepackage{bbm}
\usepackage{dsfont}
\usepackage{xcolor}   

\DeclareMathAlphabet{\mathpzc}{OT1}{pzc}{m}{it}

\newenvironment{tfae}{
\begin{enumerate}}{\end{enumerate}}

\newtheorem{prop}{Proposition}[section]
\newtheorem{lemma}[prop]{Lemma}
\newtheorem{theorem}[prop]{Theorem}
\newtheorem{corollary}[prop]{Corollary}
\theoremstyle{definition}
\newtheorem{definition}[prop]{Definition}
\newtheorem{definitions}[prop]{Definitions}
\newtheorem{example}[prop]{Example}
\newtheorem{examples}[prop]{Examples}
\newtheorem{remark}[prop]{Remark}
\newtheorem{remarks}[prop]{Remarks}

\renewcommand{\subsection}[1]{\addtocounter{subsection}{1}
\vspace*{2ex}\noindent\textbf{\thesubsection}\hspace{1ex}{\bf #1}}

\usepackage{chngcntr}
\counterwithin*{equation}{section}

\newdir{ >}{{}*!/-8pt/@{>}}

\newcommand{\fspstr}[2]{\llbracket #1,#2\rrbracket}

\def\mathrmdef#1{\expandafter\def\csname#1\endcsname{{\rm#1}}}
\mathrmdef{Aut}\mathrmdef{can}\mathrmdef{co}\mathrmdef{ev}\mathrmdef{Ker}\mathrmdef{id}\mathrmdef{Id}\mathrmdef{lex}\mathrmdef{max}\mathrmdef{op}\mathrmdef{pr}
\mathrmdef{prod}\mathrmdef{sep}

\def\mathsfdef#1{\expandafter\def\csname#1\endcsname{{\sf#1}}}
\mathsfdef{F}
\mathsfdef{P}
\mathsfdef{Gp}
\mathsfdef{oGp}
\mathsfdef{R}
\mathsfdef{U}

 \def\mathbfdef#1{\expandafter\def\csname#1\endcsname{{\rm\bf#1}}}

 \mathbfdef{C}\mathbfdef{Cat} \mathbfdef{Dist}
\mathbfdef{Met}\mathbfdef{MetGrp}\mathbfdef{Mon} \mathbfdef{MultiOrd}
\mathbfdef{Ord}\mathbfdef{RelAlg}\mathbfdef{ProbMet}\mathbfdef{ProbMetGrp}\mathbfdef{ProbUMet}\mathbfdef{ProbUMetGrp}\mathbfdef{Rel}
\mathbfdef{Set}
\mathbfdef{UMet} \mathbfdef{UMetGrp}

\newcommand{\Cats}[1]{#1\text{-}\Cat}
\newcommand{\Dists}[1]{#1\text{-}\Dist}

\newcommand{\Rels}[1]{#1\text{-}\Rel}
\newcommand{\VCat}{\Cats{V}}
\newcommand{\VCats}{\Cats{V}_\sep}
\newcommand{\VDist}{\Dists{V}}

\newcommand{\VRel}{\Rels{V}}

\def\LL{\mathbb{L}}

\def\NN{\mathbb{N}}
\def\PP{\mathbb{P}}
\def\TT{\mathbb{T}}

\def\BB{\mathbb{B}}       
\def\BBb{\mathbb{B}_\bullet} 
\def\Bb{B_\bullet}           

\def\fx{\mathfrak{x}}

\def\fX{\mathfrak{X}}

\def\relto{{\longrightarrow\hspace*{-2.8ex}{\mapstochar}\hspace*{2.6ex}}}
\def\dist{{\longrightarrow\hspace*{-3.1ex}{\circ}\hspace*{1.5ex}}}
\def\yoneda{\mathpzc{y}}
\def\mult{\mathpzc{m}}

\def\ta{\widetilde{a}}
\def\ha{\widehat{a}}
\def\hha{\widehat{\widehat{a}}}
\def\tb{\widetilde{b}}

\def\tc{\widetilde{c}}

\def\two{\mbox{\bf 2}}
\def\yk{\overline{y}}

 \newcommand{\hm}[2]{ \operatorname{hom}(#1,#2)}

 \newcommand\adjunct[2]{\xymatrix@=8ex{\ar@{}[r]|{\top}\ar@<1mm>@/^2mm/[r]^{{#2}} & \ar@<1mm>@/^2mm/[l]^{{#1}}}}

\begin{document}  
\title{On presheaf submonads of quantale-enriched categories}  
\author{Maria Manuel Clementino}

\author{Carlos Fitas}
\address{University of Coimbra, CMUC, Department of Mathematics, 3001-501 Coimbra, Portugal}
\email{mmc@mat.uc.pt, cmafitas@gmail.com}
\thanks{}

\begin{abstract}
This paper focus on the presheaf monad and its submonads on the realm of $V$-categories, for a quantale $V$. First we present two characterisations of presheaf submonads, both using $V$-distributors: one based on admissible classes of $V$-distributors, and other using Beck-Chevalley conditions on $V$-distributors. Then we focus on the study of the corresponding Eilenberg-Moore categories of algebras, having as main examples the formal ball monad and the so-called Lawvere monad.

\end{abstract}
\maketitle  

\section*{Introduction}
Having as guideline Lawvere's point of view that it is worth to regard metric spaces as categories enriched in the extended real half-line $[0,\infty]_+$ (see \cite{Law73}), we regard both the formal ball monad and the monad that identifies Cauchy complete spaces as its algebras -- which we call here the \emph{Lawvere monad} -- as submonads of the presheaf monad on the category $\Met$ of $[0,\infty]_+$-enriched categories. This leads us to the study of general presheaf submonads on the category of $V$-enriched categories, for a given quantale $V$.

Here we expand on known general characterisations of presheaf submonads and their algebras, and introduce a new ingredient -- conditions of Beck-Chevalley type -- which allows us to identify properties of functors and natural transformations, and, most importantly, contribute to a new facet of the behaviour of presheaf submonads.

In order to do that, after introducing the basic concepts needed to the study of $V$-categories in Section 1, Section 2 presents the presheaf monad and a characterisation of its submonads using admissible classes of $V$-distributors which is based on \cite{CH08}. Next we introduce the already mentioned Beck-Chevalley conditions (BC*) which resemble those discussed in \cite{CHJ14}, with $V$-distributors playing the role of $V$-relations. In particular we show that lax idempotency of a monad $\TT$ on $\VCat$ can be identified via a BC* condition, and that the presheaf monad satisfies fully BC*.
This leads to the use of BC* to present a new characterisation of presheaf submonads in Section 4.

The remaining sections are devoted to the study of the Eilenberg-Moore category induced by presheaf submonads. In Section 5, based on \cite{CH08}, we detail the relationship between the algebras, (weighted) cocompleteness, and injectivity. Next we focus on the algebras and their morphisms, first for the formal ball monad, and later for a general presheaf submonad. We end by presenting the relevant example of the presheaf submonad whose algebras are the so-called Lawvere complete $V$-categories \cite{CH09}, which, when $V=[0,\infty]_+$, are exactly the Cauchy complete (generalised) metric spaces, while their morphisms are the $V$-functors which preserve the limits for Cauchy sequences.

\section{Preliminaries}

Our work focus on $V$-categories (or $V$-enriched categories, cf. \cite{EK66, Law73, Kel82}) in the special case of $V$ being a quantale.

Throughout $V$ is a \emph{commutative and unital quantale}; that is, $V$ is a complete lattice endowed with a symmetric tensor product $\otimes$, with unit $k\neq\bot$, commuting with joins, so that it has a right adjoint $\hom$; this means that, for $u,v,w\in V$,
\[u\otimes v\leq w\;\Leftrightarrow\; v\leq\hom(u,w).\]
As a category, $V$ is a complete and cocomplete (thin) symmetric monoidal closed category.

\begin{definition}
A \emph{$V$-category} is a pair $(X,a)$ where $X$ is a set and $a\colon X\times X\to V$ is a map such that:
\begin{itemize}
\item[(R)] for each $x\in X$, $k\leq a(x,x)$;
\item[(T)] for each $x,x',x''\in X$, $a(x,x')\otimes a(x',x'')\leq a(x,x'')$.
\end{itemize}
If $(X,a)$, $(Y,b)$ are $V$-categories, a \emph{$V$-functor} $f\colon (X,a)\to(Y,b)$ is a map $f\colon X\to Y$ such that
\begin{itemize}
\item[(C)] for each $x,x'\in X$, $a(x,x')\leq b(f(x),f(x'))$.
\end{itemize}
The category of $V$-categories and $V$-functors will be denoted by $\VCat$. Sometimes we will use the notation $X(x,y)=a(x,y)$ for a $V$-category $(X,a)$ and $x,y\in X$.
\end{definition}

We point out that $V$ has itself a $V$-categorical structure, given by the right adjoint to $\otimes$, $\hom$; indeed, $u\otimes k\leq u\;\Rightarrow\;k\leq\hom(u,u)$, and $u\otimes\hom(u,u')\otimes\hom(u',u'')\leq u'\otimes\hom(u',u'')\leq u''$ gives that $\hom(u,u')\otimes\hom(u',u'')\leq\hom(u,u'')$. Moreover, for every $V$-category $(X,a)$, one can define its \emph{opposite $V$-category} $(X,a)^\op=(X,a^\circ)$, with $a^\circ(x,x')=a(x',x)$ for all $x,x'\in X$.

\begin{examples}\label{ex:VCat}
\begin{enumerate}
\item For $V=\two=(\{0<1\},\wedge,1)$, a $\two$-category is an
    \emph{ordered set} (not necessarily antisymmetric) and a
    $\two$-functor is a \emph{monotone map}. We denote $\Cats{\two}$ by $\Ord$.
  \item The lattice $V=[0,\infty]$ ordered by the ``greater
    or equal'' relation $\geq$ (so that $r\wedge s=\max\{r,s\}$, and the supremum of $S\subseteq[0,\infty]$ is given
    by $\inf S$) with tensor $\otimes=+$ will be denoted by $[0,\infty]_+$. A $[0,\infty]_+$-category is a
    \emph{(generalised) metric space} and a
    $[0,\infty]_+$-functor is a \emph{non-expansive map}  (see \cite{Law73}). We denote $\Cats{[0,\infty]_+}$ by $\Met$. We note that
    \[
      \hom(u,v)=v\ominus u:=\max\{v-u,0\},
    \]
    for all $u,v\in[0,\infty]$.

    If instead of $+$ one considers the tensor product $\wedge$, then $\Cats{[0,\infty]_\wedge}$ is the category $\UMet$ of \emph{ultrametric spaces} and \emph{non-expansive maps}.
  \item\label{ex.zero_um} The complete lattice $[0,1]$ with the usual
    ``less or equal'' relation $\le$ is isomorphic to $[0,\infty]$ via
    the map $[0,1]\to[0,\infty],\,u\mapsto -\ln(u)$ where
    $-\ln(0)=\infty$. Under this isomorphism, the operation $+$ on
    $[0,\infty]$ corresponds to the multiplication $*$ on $[0,1]$. Denoting this quantale by $[0,1]_*$, one has $\Cats{[0,1]_*}$
    isomorphic to the category $\Met=\Cats{[0,\infty]_+}$ of (generalised) metric spaces and
    non-expansive maps.

    Since $[0,1]$ is a frame, so that finite meets commute with infinite joins, we can also consider it as a quantale
    with $\otimes=\wedge$. The category $\Cats{[0,1]_{\wedge}}$ is isomorphic
    to the category $\UMet$.

    Another interesting tensor product in $[0,1]$ is given by the
    \emph{\L{}ukasiewicz tensor} $\odot$ where $u\odot v=\max(0,u+v-1)$; here
    $\hom(u,v)=\min(1,1-u+v)$. Then $\Cats{[0,1]_\odot}$ is the category of \emph{bounded-by-1 (generalised) metric spaces} and \emph{non-expansive maps}.

  \item We consider now the set
    \[
      \Delta=\{\varphi\colon[0,\infty]\to [0,1]\mid \text{for all
        $\alpha\in[0,\infty]$: }\varphi(\alpha)=\bigvee_{\beta<\alpha}
      \varphi(\beta) \},
    \]
    of \emph{distribution functions}. With the
    pointwise order, it is a complete lattice. For $\varphi,\psi\in\Delta$ and
    $\alpha\in[0,\infty]$, define $\varphi\otimes\psi\in\Delta$ by
    \[
      (\varphi\otimes \psi)(\alpha)=\bigvee_{\beta+\gamma\le
        \alpha}\varphi(\beta)*\psi(\gamma).
    \]
    Then
    $\otimes:\Delta\times\Delta\to\Delta$ is associative and
    commutative, and
    \[
      \kappa:[0,\infty]\to [0,1],\, \alpha\mapsto
      \begin{cases}
        0 & \text{if }\alpha=0,\\
        1 & \text{else}
      \end{cases}
    \]
    is a unit for $\otimes$. Finally,
    $\psi\otimes-:\Delta\to\Delta$ preserves suprema since, for all $u\in [0,1]$,
    $u*-\colon[0,1]\to[0,1]$ preserves suprema. A $\Delta$-category is a
    \emph{(generalised) probabilistic metric space} and a
    $\Delta$-functor is a \emph{probabilistic non-expansive
      map} (see \cite{HR13} and references there).
  \end{enumerate}
\end{examples}

We will also make use of two additional categories we describe next, the category $\VRel$, of sets and $V$-relations, and the category $\VDist$, of $V$-categories and $V$-distributors.\\

Objects of $\VRel$ are sets, while morphisms are $V$-relations, i.e., if $X$ and $Y$ are sets, a \emph{$V$-relation} $r\colon X\relto Y$ is a map $r\colon X\times Y\to V$. Composition of $V$-relations is given by \emph{relational composition}, so that the composite of $r\colon X\relto Y$ and $s\colon Y\relto Z$ is given by
\[(s\cdot r)(x,z)=\bigvee_{y\in Y}r(x,y)\otimes s(y,z),\]
for every $x\in X$, $z\in Z$.
Identities in $\VCat$ are simply identity relations, with $1_X(x,x')=k$ if $x=x'$ and $1_X(x,x')=\bot$ otherwise.
The category $\VRel$ has an involution $(\;)^\circ$, assigning to each $V$-relation $r\colon X\relto Y$ the $V$-relation $r^\circ\colon Y\relto X$ defined by $r^\circ(y,x)=r(x,y)$, for every $x\in X$, $y\in Y$.

Since every map $f\colon X\to Y$ can be thought as a $V$-relation through its graph $f_\circ\colon X\times Y\to V$, with $f_\circ(x,y)=k$ if $f(x)=y$ and $f_\circ(x,y)=\bot$ otherwise, there is an injective on objects and faithful functor $\Set\to\VRel$. When no confusion may arise, we use also $f$ to denote the $V$-relation $f_\circ$.

The category $\VRel$ is a 2-category, when equipped with the 2-cells given by the pointwise order; that is, for $r,r'\colon X\relto Y$, one defines $r\leq r'$ if, for all $x\in X$, $y\in Y$, $r(x,y)\leq r'(x,y)$. This gives us the possibility of studying adjointness between $V$-relations. We note in particular that, if $f\colon X\to Y$ is a map, then $f_\circ\cdot f^\circ\leq 1_Y$ and $1_X\leq f^\circ\cdot f_\circ$, so that $f_\circ\dashv f^\circ$.\\

Objects of $\VDist$ are $V$-categories, while morphisms are $V$-distributors (also called $V$-bimodules, or $V$-profunctors); i.e., if $(X,a)$ and $(Y,b)$ are $V$-categories, a \emph{$V$-distributor} -- or, simply, a \emph{distributor} -- $\varphi\colon (X,a)\dist (Y,b)$ is a $V$-relation $\varphi\colon X\relto Y$ such that $\varphi\cdot a\leq\varphi$ and $b\cdot\varphi\leq\varphi$ (in fact $\varphi\cdot a=\varphi$ and $b\cdot\varphi=\varphi$ since the other inequalities follow from (R)). Composition of distributors is again given by relational composition, while the identities are given by the $V$-categorical structures, i.e. $1_{(X,a)}=a$. Moreover, $\VDist$ inherits the 2-categorical structure from $\VRel$.\\

Each $V$-functor $f\colon(X,a)\to(Y,b)$ induces two distributors, $f_*\colon(X,a)\dist(Y,b)$ and \linebreak $f^*\colon(Y,b)\dist(X,a)$, defined by $f_*(x,y)=Y(f(x),y)$ and $f^*(y,x)=Y(y,f(x))$, that is, $f_*=b\cdot f_\circ$ and $f^*=f^\circ\cdot b$. These assignments are functorial, as we explain below.\\

First we define 2-cells in $\VCat$: for $f,f'\colon(X,a)\to(Y,b)$ $V$-functors, $f\leq f'$ when $f^*\leq (f')^*$ as distributors, so that
\[f\leq f' \;\;\Leftrightarrow\;\; \forall x\in X, \,y\in Y, \; Y(y,f(x))\leq Y(y,f'(x)).\] $\VCat$ is then a 2-category, and we can define two 2-functors
\[\begin{array}{rclcrcl}
(\;)_*\colon \VCat^\co&\longrightarrow&\VDist&\mbox{ and }&(\;)^*\colon\VCat^\op&\longrightarrow&\VDist\\
X&\longmapsto&X&&X&\longmapsto&X\\
f&\longmapsto&f_*&&f&\longmapsto&f^*
\end{array}\]
Note that, for any $V$-functor $f\colon(X,a)\to(Y,b)$,
\[f_*\cdot f^*=b\cdot f_\circ\cdot f^\circ\cdot b\leq b\cdot b=b\mbox{ and }f^*\cdot f_*=f^\circ\cdot b\cdot b\cdot f_\circ\geq f^\circ\cdot f_\circ\cdot a\geq a;\]
hence every $V$-functor induces a pair of adjoint distributors, $f_*\dashv f^*$. A $V$-functor $f\colon X\to Y$ is said to be \emph{fully faithful} if $f^*\cdot f_*=a$, i.e. $X(x,x')=Y(f(x),f(x'))$ for all $x,x'\in X$, while it is \emph{fully dense} if $f_*\cdot f^*=b$, i.e. $Y(y,y')=\bigvee_{x\in X}Y(y,f(x))\otimes Y(f(x),y')$, for all $y,y'\in Y$. A fully faithful $V$-functor $f\colon X\to Y$ does not need to be an injective map; it is so in case $X$ and $Y$ are separated $V$-categories (as defined below).\\

\begin{remark}\label{rem:adjcond}
In $\VCat$ adjointness between $V$-functors
\[Y\adjunct{f}{g}X\]
 can be equivalently expressed as:
\[f\dashv g\;\Leftrightarrow\;f_*=g^*\;\Leftrightarrow\; g^*\dashv f^* \;\Leftrightarrow\;(\forall x\in X)\;(\forall y\in Y)\;X(x,g(y))=Y(f(x),y).\]
In fact the latter condition encodes also $V$-functoriality of $f$ and $g$; that is, if $f\colon X\to Y$ and $g\colon Y\to X$ are maps satisfying the condition
\[(\forall x\in X)\;(\forall y\in Y)\;\;  X(x,g(y))=Y(f(x),y),\]
then $f$ and $g$ are $V$-functors, with $f\dashv g$.

Furthermore, it is easy to check that, given $V$-categories $X$ and $Y$, a map $f\colon X\to Y$ is a $V$-functor whenever $f_*$ is a distributor (or whenever $f^*$ is a distributor).
\end{remark}

The order defined on $\VCat$ is in general not antisymmetric. For $V$-functors $f,g\colon X\to Y$, one says that $f\simeq g$ if $f\leq g$ and $g\leq f$ (or, equivalently, $f^*=g^*$). For elements $x,y$ of a $V$-category $X$, one says that $x\leq y$ if, considering the $V$-functors $x,y\colon E=(\{*\},k)\to X$ (where $k(*,*)=k$) defined by $x(*)=x$ and $y(*)=y$, one has $x\leq y$; or, equivalently, $X(x,y)\geq k$. Then, for any $V$-functors $f,g\colon X\to Y$, $f\leq g$ if, and only if, $f(x)\leq g(x)$ for every $x\in X$.

\begin{definition}
A $V$-category $Y$ is said to be \emph{separated} if, for $f,g\colon X\to Y$, $f=g$ whenever $f\simeq g$; equivalently, if, for all $x,y\in Y$, $x\simeq y$ implies $x=y$.
\end{definition}

The tensor product $\otimes$ on $V$ induces a tensor product on $\VCat$, with $(X,a)\otimes(Y,b)=(X\times Y,a\otimes b)=X\otimes Y$, where $(X\otimes Y)((x,y),(x',y'))=X(x,x')\otimes Y(y,y')$. The $V$-category $E$ is a $\otimes$-neutral element. With this tensor product, $\VCat$ becomes a monoidal closed category. Indeed, for each $V$-category $X$, the functor $X\otimes (\;)\colon \VCat\to\VCat$ has a right adjoint $(\;)^X$ defined by $Y^X=(\VCat(X,Y), \fspstr{\;}{\;} )$, with $\fspstr{f}{g}=\bigwedge_{x\in X}Y(f(x),g(x))$ (see \cite{EK66, Law73, Kel82} for details).

It is interesting to note the following well-known result (see, for instance, \cite[Theorem 2.5]{CH09}).

\begin{theorem}\label{th:fct_dist}
For $V$-categories $(X,a)$ and $(Y,b)$, and a $V$-relation $\varphi\colon X\relto Y$, the following conditions are equivalent:
\begin{tfae}
\item $\varphi\colon(X,a)\dist(Y,b)$ is a distributor;
\item $\varphi\colon(X,a)^\op\otimes(Y,b)\to(V,\hom)$ is a $V$-functor.
\end{tfae}
\end{theorem}

In particular, the $V$-categorical structure $a$ of $(X,a)$ is a $V$-distributor $a\colon(X,a)\dist(X,a)$, and therefore a $V$-functor $a\colon(X,a)^\op\otimes (X,a)\to(V,\hom)$, which induces, via the closed monoidal structure of $\VCat$, the \emph{Yoneda $V$-functor} $\yoneda_X\colon(X,a)\to (V,\hom)^{(X,a)^\op}$. Thanks to the theorem above, $V^{X^\op}$ can be equivalently described as
\[PX:=\{\varphi\colon X \dist E\,|\,\varphi \mbox{ $V$-distributor}\}.\]
Then the structure $\ta$ on $PX$ is given by \[\ta(\varphi,\psi)=\fspstr{\varphi}{\psi}=\bigwedge_{x\in X}\hom(\varphi(x),\psi(x)),\] for every $\varphi, \psi\colon X\dist E$, where by $\varphi(x)$ we mean $\varphi(x,*)$, or, equivalently, we consider the associated $V$-functor $\varphi\colon X\to V$. The Yoneda functor $\yoneda_X\colon X\to PX$ assigns to each $x\in X$ the distributor $x^*\colon X\dist E$, where we identify again $x\in X$ with the $V$-functor $x\colon E\to X$ assigning $x$ to the (unique) element of $E$. Then, for every $\varphi\in PX$ and $x\in X$, we have that \[\fspstr{\yoneda_X(x)}{\varphi}=\varphi(x),\]
as expected. In particular $\yoneda_X$ is a fully faithful $V$-functor, being injective on objects (i.e. an injective map) when $X$ is a separated $V$-category. We point out that $(V,\hom)$ is separated, and so is $PX$ for every $V$-category $X$.

For more information on $\VCat$ we refer to \cite[Appendix]{HN20}.

\section{The presheaf monad and its submonads}

The assignment $X\mapsto PX$ defines a functor $P\colon\VCat\to\VCat$: for each $V$-functor $f\colon X\to Y$, $Pf\colon PX\to PY$ assigns to each distributor $\xymatrix{X\ar[r]|{\circ}^{\varphi}&E}$ the distributor $\xymatrix{Y\ar[r]|{\circ}^{f^*}&X\ar[r]|{\circ}^{\varphi}&E}$. It is easily checked that the Yoneda functors $(\yoneda_X\colon X\to PX)_X$ define a natural transformation $\yoneda\colon 1\to P$. Moreover, since, for every $V$-functor $f$, the adjunction $f_*\dashv f^*$ yields an adjunction $Pf=(\;)\cdot f^*\dashv (\;)\cdot f_*=:Qf$, $P\yoneda_X$ has a right adjoint, which we denote by $\mult_X\colon PPX\to PX$. It is straightforward to check that $\PP=(P,\mult,\yoneda)$ is a 2-monad on $\VCat$ -- the so-called \emph{presheaf monad} --, which, by construction of $\mult_X$ as the right adjoint to $P\yoneda_X$, is lax idempotent (see \cite{Ho11} for details).\\

Next we present a characterisation of the submonads of $\PP$ which is partially in \cite{CH08}. We recall that, given two monads $\TT=(T,\mu,\eta)$, $\TT'=(T',\mu',\eta')$ on a category $\C$, a monad morphism $\sigma\colon\TT\to\TT'$ is a natural transformation $\sigma\colon T\to T'$ such that
\begin{equation}\label{eq:monadmorphism}
\xymatrix{1\ar[r]^{\eta}\ar[rd]_{\eta'}&T\ar[d]^{\sigma}&&TT\ar[r]^-{\sigma_T}\ar[d]_{\mu}&T'T\ar[r]^-{T'\sigma}&T'T'\ar[d]^{\mu'}\\
&T'&&T\ar[rr]_{\sigma}&&T'}
\end{equation}
By \emph{submonad of $\PP$} we mean a 2-monad $\TT=(T,\mu,\eta)$ on $\VCat$ with a monad morphism $\sigma:\TT\to\PP$ such that $\sigma_X$ is an embedding (i.e. both fully faithful and injective on objects) for every $V$-category $X$.

\begin{definition}\label{def:admi}
Given a class $\Phi$ of $V$-distributors, for every $V$-category $X$ let \[\Phi X=\{\varphi\colon X\dist E\,|\,\varphi\in\Phi\}\] have the $V$-category structure inherited from the one of $PX$. We say that $\Phi$ is \emph{admissible} if, for every $V$-functor $f\colon X\to Y$ and $V$-distributors $\varphi\colon Z\dist Y$ and $\psi\colon X\dist Z$ in $\Phi$,
\begin{itemize}
\item[(1)] $f^*\in\Phi$;
\item[(2)] $\psi\cdot f^*\in\Phi$ and $f^*\cdot \varphi\in\Phi$;
\item[(3)] $\varphi\in\Phi\;\Leftrightarrow\;(\forall y\in Y)\;y^*\cdot\varphi\in\Phi$;
\item[(4)] for every $V$-distributor $\gamma\colon PX\dist E$, if the restriction of $\gamma$ to $\Phi X$ belongs to $\Phi$, then $\gamma\cdot(\yoneda_X)_*\in\Phi$.
\end{itemize}
\end{definition}

\begin{lemma}
Every admissible class $\Phi$ of $V$-distributors induces a submonad $\Phi=(\Phi,\mult^\Phi,\yoneda^\Phi)$ of $\PP$.
\end{lemma}
\begin{proof}
For each $V$-category $X$, equip $\Phi X$ with the initial structure induced by the inclusion $\sigma_X\colon \Phi X\to PX$, that is, for every $\varphi,\psi\in \Phi X$, $\Phi X(\varphi,\psi)=PX(\varphi,\psi)$. For each $V$-functor $f\colon X\to Y$ and $\varphi\in\Phi X$, by condition (2), $\varphi\cdot f^*\in\Phi$, and so $Pf$ (co)restricts to $\Phi f\colon\Phi X\to\Phi Y$.

Condition (1) guarantees that $\yoneda_X\colon X\to PX$ corestricts to $\yoneda^\Phi_X\colon X\to \Phi X$.

Finally, condition (4) guarantees that $\mult_X\colon PPX\to PX$ also (co)restricts to $\mult^\Phi_X:\Phi\Phi X\to\Phi X$: if $\gamma\colon\Phi X\dist E$ belongs to $\Phi$, then $\widetilde{\gamma}:=\gamma\cdot (\sigma_X)^*\colon PX\dist E$ belongs to $\Phi$ by (2), and then, since $\gamma$ is the restriction of $\widetilde{\gamma}$ to $\Phi X$, by (4) $\mult_X(\widetilde{\gamma})=\gamma\cdot(\sigma_X)^*\cdot(\yoneda_X)_*
=\gamma\cdot(\sigma_X)^*\cdot(\sigma_X)_*\cdot(\yoneda^\Phi_X)_*=\gamma\cdot(\yoneda^\Phi_X)_*\in\Phi$.

By construction, $(\sigma_X)_X$ is a natural transformation, each $\sigma_X$ is an embedding, and $\sigma$ makes diagrams \eqref{eq:monadmorphism} commute.
\end{proof}

\begin{theorem}\label{th:Phi}
For a 2-monad $\TT=(T,\mu,\eta)$ on $\VCat$, the following assertions are equivalent:
\begin{tfae}
\item $\TT$ is isomorphic to $\Phi$, for some admissible class of $V$-distributors $\Phi$.
\item $\TT$ is a submonad of $\PP$.
\end{tfae}
\end{theorem}

\begin{proof}
(i) $\Rightarrow$ (ii) follows from the lemma above.\\

(ii) $\Rightarrow$ (i): Let $\sigma\colon\TT\to\PP$ be a monad morphism, with $\sigma_X$ an embedding for every $V$-category $X$, which, for simplicity, we assume to be an inclusion. First we show that
\begin{equation}\label{eq:fai}
\Phi=\{\varphi\colon X\dist Y\,|\,\forall y\in Y\;y^*\cdot\varphi\in TX\}
\end{equation}
is admissible. In the sequel $f\colon X\to Y$ is a $V$-functor.\\

(1) For each $x\in X$, $x^*\cdot f^*=f(x)^*\in TY$, and so $f^*\in\Phi$.\\

(2) If $\psi\colon X\dist Z$ is a $V$-distributor in $\Phi$, and $z\in Z$, since $z^*\cdot\psi\in TX$, $T f(z^*\cdot\psi)=z^*\cdot\psi\cdot f^*\in TY$, and therefore $\psi\cdot f^*\in \Phi$ by definition of $\Phi$.
Now, if $\varphi\colon Z\dist Y\in\Phi$, then, for each $x\in X$, $x^*\cdot f^*\cdot\varphi=f(x)^*\cdot\varphi\in TZ$ because $\varphi\in\Phi$, and so $f^*\cdot\varphi\in\Phi$.\\

(3) follows from the definition of $\Phi$.\\

(4) If the restriction of $\gamma\colon PX\dist E$ to $TX$, i.e., $\gamma\cdot(\sigma_X)_*$, belongs to $\Phi$, then $\mu_X(\gamma\cdot(\sigma_X)_*)=\gamma\cdot(\sigma_X)_*\cdot(\eta_X)_*=\gamma\cdot(\yoneda_X)_*$ belongs to $TX$.
\end{proof}

We point out that, with $\PP$, also $\TT$ is lax idempotent. This assertion is shown at the end of next section, making use of the Beck-Chevalley conditions we study next. (We note that the arguments of \cite[Prop. 16.2]{CLF20}, which states conditions under which a submonad of a lax idempotent monad is still lax idempotent, cannot be used directly here.)

\section{The presheaf monad and Beck-Chevalley conditions}

In this section our aim is to show that $\PP$ verifies some interesting conditions of Beck-Chevalley type, that resemble the BC conditions studied in \cite{CHJ14}. We recall from \cite{CHJ14} that a commutative square in $\Set$
\[\xymatrix{W\ar[r]^l\ar[d]_g&Z\ar[d]^h\\
X\ar[r]_f&Y}\]
is said to be a \emph{BC-square} if the following diagram commutes in $\Rel$
\[\xymatrix{W\ar[r]|-{\object@{|}}^{l_\circ}&Z\\
X\ar[u]|-{\object@{|}}^{g^\circ}\ar[r]|-{\object@{|}}_{f_\circ}&Y,\ar[u]|-{\object@{|}}_{h^\circ}}\]
where, given a map $t\colon A \to B$, $t_\circ\colon A\relto B$ denotes the relation defined by $t$ and $t^\circ\colon B\relto A$ its opposite. Since $t_\circ\dashv t^\circ$ in $\Rel$, this is in fact a kind of Beck-Chevalley condition. A $\Set$-endofunctor $T$ is said to satisfy BC if it preserves BC-squares, while a natural transformation $\alpha\colon T\to T'$ between two $\Set$-endofunctors satisfies BC if, for each map $f\colon X\to Y$, its naturality square
\[\xymatrix{TX\ar[r]^{\alpha_X}\ar[d]_{Tf}&T'X\ar[d]^{T'f}\\
TY\ar[r]_{\alpha_Y}&T'Y}\]
is a BC-square.

In our situation, for endofunctors and natural transformations in $\VCat$, the role of $\Rel$ is played by $\VDist$.

\begin{definition}
A commutative square in $\VCat$
\[\xymatrix{(W,d)\ar[r]^l\ar[d]_g&(Z,c)\ar[d]^h\\
(X,a)\ar[r]_f&(Y,b)}\]
is said to be a \emph{BC*-square} if the following diagram commutes in $\VDist$
\begin{equation}\label{diag:BC*}
\xymatrix{(W,d)\ar[r]|-{\circ}^{l_*}&(Z,c)\\
(X,a)\ar[u]|-{\circ}^{g^*}\ar[r]|-{\circ}_{f_*}&(Y,b)\ar[u]|-{\circ}_{h^*}}\end{equation}
(or, equivalently, $h^*\cdot f_*\leq l_*\cdot g^*$).
\end{definition}

\begin{remarks}\label{rem:BC*}
\begin{enumerate}
\item For a $V$-functor $f\colon(X,a)\to(Y,b)$, to be fully faithful is equivalent to
\[\xymatrix{(X,a)\ar[r]^1\ar[d]_1&(X,a)\ar[d]^f\\
(X,a)\ar[r]_f&(Y,b)}\]
being a BC*-square (exactly in parallel with the characterisation of monomorphisms via BC-squares).
\item We point out that, contrarily to the case of BC-squares,  in BC*-squares the horizontal and the vertical arrows play different roles; that is, the fact that diagram \eqref{diag:BC*} is a BC*-square is not equivalent to
\[\xymatrix{(W,d)\ar[r]^g\ar[d]_l&(X,a)\ar[d]^f\\
(Z,c)\ar[r]_h&(Y,b)}\]
being a BC*-square; it is indeed equivalent to its \emph{dual}
\[\xymatrix{(W,d^\circ)\ar[r]^g\ar[d]_l&(X,a^\circ)\ar[d]^f\\
(Z,c^\circ)\ar[r]_h&(Y,b^\circ)}\]
being a BC*-square.
\end{enumerate}
\end{remarks}

\begin{definitions}
\begin{enumerate}
\item A \emph{functor $T\colon \VCat\to\VCat$ satisfies BC*} if it preserves BC*-squares.
\item Given two endofunctors $T,T'$ on $\VCat$, a \emph{natural transformation $\alpha\colon T\to T'$ satisfies BC*} if the naturality diagram
    \[\xymatrix{TX\ar[r]^{\alpha_X}\ar[d]_{Tf}&T'X\ar[d]^{T'f}\\
TY\ar[r]_{\alpha_Y}&T'Y}\]
is a BC*-square for every morphism $f$ in $\VCat$.
\item \emph{A 2-monad $\TT=(T,\mu,\eta)$} on $\VCat$ is said to satisfy \emph{fully BC*} if $T$, $\mu$, and $\eta$ satisfy BC*.
\end{enumerate}
\end{definitions}

\begin{remark}
In the case of $\Set$ and $\Rel$, since the condition of being a BC-square is equivalent, under the Axiom of Choice (AC), to being a weak pullback, a $\Set$-monad $\TT$ \emph{satisfies fully BC} if, and only if, it is \emph{weakly cartesian} (again, under (AC)). This, together with the fact that there are relevant $\Set$-monads -- like for instance the ultrafilter monad -- whose functor and multiplication satisfy BC but the unit does not, led the authors of \cite{CHJ14} to name such monads as \emph{BC-monads}. This is the reason why we use \emph{fully BC*} instead of BC* to identify these $\VCat$-monads.

As a side remark we recall that, still in the $\Set$-context, a partial BC-condition was studied by Manes in \cite{Ma02}: for a $\Set$-monad $\TT=(T,\mu,\eta)$ to be \emph{taut} requires that $T$, $\mu$, $\eta$ satisfy BC for commutative squares where $f$ is monic.
\end{remark}

Our first use of BC* is the following characterisation of lax idempotency for a 2-monad $\TT$ on $\VCat$.

\begin{prop}\label{prop:laxidpt} Let $\TT=(T,\mu,\eta)$ be a 2-monad on $\VCat$.
\begin{enumerate}
\item The following assertions are equivalent:
\begin{tfae}
\item[\em (i)] $\TT$ is lax idempotent.
\item[\em (ii)] For each $V$-category $X$, the diagram
\begin{equation}\label{eq:laxidpt}
\xymatrix{TX\ar[r]^-{T\eta_X}\ar[d]_{\eta_{TX}}&TTX\ar[d]^{\mu_X}\\
TTX\ar[r]_-{\mu_X}&TX}
\end{equation}
is a BC*-square.
\end{tfae}
\item If $\TT$ is lax idempotent, then $\mu$ satisfies BC*.
\end{enumerate}
\end{prop}
\begin{proof}
(1) (i) $\Rightarrow$ (ii): The monad $\TT$ is lax idempotent if, and only if, for every $V$-category $X$, $T\eta_X\dashv \mu_X$, or, equivalently, $\mu_X\dashv \eta_{TX}$. These two conditions are equivalent to $(T\eta_X)_*=(\mu_X)^*$ and $(\mu_X)_*=(\eta_{TX})^*$. Hence $(\mu_X)^*(\mu_X)_*=(T\eta_X)_* (\eta_{TX})^*$ as claimed.

(ii) $\Rightarrow$ (i): From $(\mu_X)^* (\mu_X)_*=(T\eta_X)_* (\eta_{TX})^*$ it follows that
\[(\mu_X)_*=(\mu_X)_* (\mu_X)^* (\mu_X)_*=(\mu_X\cdot T\eta_X)_* (\eta_{TX})^*=(\eta_{TX})^*,\]
that is, $\mu_X\dashv \eta_{TX}$.\\

(2) BC* for $\mu$ follows directly from lax idempotency of $\TT$, since
\[\xymatrix{TTX\ar[r]^-{(\mu_X)_*}|-{\circ}&TX\ar@{}[rrd]|{=}&&TTX\ar[r]^-{(\eta_{TX})^*}|-{\circ}&TX\\
TTY\ar[u]^{(TTf)^*}|-{\circ}\ar[r]_-{(\mu_Y)_*}|-{\circ}&TY\ar[u]_{(Tf)^*}|-{\circ}&&
TTY\ar[u]^{(TTf)^*}|-{\circ}\ar[r]_-{(\eta_{TY})^*}|-{\circ}&TY\ar[u]_{(Tf)^*}|-{\circ}}\]
and the latter diagram commutes trivially.\\
\end{proof}

\begin{remark}
Thanks to Remarks \ref{rem:BC*} we know that, if we invert the role of $\eta_{TX}$ and $T\eta_X$ in \eqref{eq:laxidpt}, we get a characterisation of oplax idempotent 2-monad: $\TT$ is oplax idempotent if, and only if, the diagram
\[\xymatrix{TX\ar[r]^-{\eta_{TX}}\ar[d]_{T\eta_{X}}&TTX\ar[d]^{\mu_X}\\
TTX\ar[r]_-{\mu_X}&TX}\]
is a BC*-square.
\end{remark}

\begin{theorem}
The presheaf monad $\PP=(P,\mult,\yoneda)$ satisfies fully BC*.
\end{theorem}

\begin{proof}
(1) \emph{$P$ satisfies BC*}: Given a BC*-square
\[\xymatrix{(W,d)\ar[r]^l\ar[d]_g&(Z,c)\ar[d]^h\\
(X,a)\ar[r]_f&(Y,b)}\]
in $\VCat$, we want to show that
\begin{equation}\label{eq:BC}
\xymatrix{PW\ar[r]|-{\circ}^{(Pl)_*}&PZ\\
PX\ar@{}[ru]|{\geq}\ar[u]|-{\circ}^{(Pg)^*}\ar[r]|-{\circ}_{(Pf)_*}&PY.\ar[u]|-{\circ}_{(Ph)^*}}
\end{equation}
For each $\varphi\in PX$ and $\psi\in PZ$, we have
\begin{align*}
(Ph)^*(Pf)_*(\varphi,\psi)&= (Ph)^\circ\cdot\tb\cdot Pf(\varphi,\psi)\\
&=\tb(Pf(\varphi),Ph(\psi))\\
&= \bigwedge_{y\in Y}\hom(\varphi\cdot f^*(y),\psi\cdot h^*(y))\\
&\leq \displaystyle\bigwedge_{x\in X} \hom(\varphi\cdot f^*\cdot f_*(x),\psi\cdot h^*\cdot f_*(x))\\
&\leq \displaystyle\bigwedge_{x\in X}\hom(\varphi(x),\psi\cdot l_*\cdot g^*(x))&\mbox{($\varphi\leq \varphi\cdot f^*\cdot f_*$, \eqref{eq:BC} is BC*)}\\
&= \ta(\varphi,\psi\cdot\l_*\cdot g^*)\\
&\leq \ta(\varphi,\psi\cdot l_*\cdot g^*)\otimes \tc(\psi\cdot l_*\cdot l^*,\psi)&\mbox{(because $\psi\cdot l_*\cdot l^*\leq\psi$)}\\
&=\ta(\varphi,Pg(\psi\cdot l_*)\otimes\tc(Pl(\psi\cdot l_*),\psi)\\
&\leq \displaystyle\bigvee_{\gamma\in PW}\ta(\varphi,Pg(\gamma))\otimes\tc(Pl(\gamma),\psi)\\
&=(Pl)_*(Pg)^*(\varphi,\psi).
\end{align*}

(2) \emph{$\mu$ satisfies BC*}: For each $V$-functor $f\colon X\to Y$, from the naturality of $\yoneda$ it follows that the following diagram
\[\xymatrix{PPX\ar[r]|-{\circ}^-{(\yoneda_{PX})^*}&PX\\
PPY\ar[u]|-{\circ}^{(PPf)^*}\ar[r]|-{\circ}_-{(\yoneda_{PY})^*}&PY\ar[u]|-{\circ}_{(Pf)^*}}\]
commutes. Lax idempotency of $\PP$ means in particular that $\mult_X\dashv \yoneda_{PX}$, or, equivalently, $(\mult_X)_*=(\yoneda_{PX})^*$, and therefore the commutativity of this diagram shows BC* for $\mult$.

(3) \emph{$\yoneda$ satisfies BC*}: Once again, for each $V$-functor $f\colon(X,a)\to(Y,b)$, we want to show that the diagram
\[\xymatrix{X\ar[r]|-{\circ}^-{(\yoneda_X)_*}&PX\\
Y\ar[u]|-{\circ}^{f^*}\ar[r]|-{\circ}_-{(\yoneda_Y)_*}&PY\ar[u]|-{\circ}_{(Pf)^*}}\]
commutes. Let $y\in Y$ and $\varphi\colon X\dist E$ belong to $PX$. Then
\begin{align*}((Pf)^*(\yoneda_Y)_*)(y,\varphi)&=((Pf)^\circ\cdot \tb\cdot\yoneda_Y)(y,\varphi)=\tb(\yoneda_Y(y),Pf(\varphi))=Pf(\varphi)(y)=\bigvee_{x\in X}b(y,f(x))\otimes\varphi(x)\\
&=\bigvee_{x\in X}b(y,f(x))\otimes\ta(\yoneda_X(x),\varphi)=(\ta\cdot\yoneda_X\cdot f^\circ\cdot b)(y,\varphi)=(\yoneda_X)_*\cdot f^*(y,\varphi),\\
\end{align*}
as claimed.
\end{proof}

\begin{corollary}\label{cor:laxidpt}
Let  $\TT=(T,\mu,\eta)$ on $\VCat$ be a 2-monad on $\VCat$, and $\sigma\colon\TT\to\PP$ be a monad morphism, pointwise fully faithful. Then $\TT$ is lax idempotent.
\end{corollary}

\begin{proof}
We know that $\PP$ is lax idempotent, and so, for every $V$-category $X$, $(\mult_X)_*=(\yoneda_{PX})^*$.
Consider diagram \eqref{eq:monadmorphism}. The commutativity of the diagram on the right gives that $(\mu_X)_*=(\sigma_X)^*(\sigma_X)_*(\mu_X)_*=(\sigma_X)^*(\mult_X)_*(P\sigma_X)_*(\sigma_{TX})_*$; using the equality above, and preservation of fully faithful $V$-functors by $\PP$ -- which follows from BC* -- we obtain:
\begin{align*}
(\mu_X)_*&=(\sigma_X)^*(\yoneda_{PX})^*(P\sigma_X)_*(\sigma_{TX})_*=(\sigma_X)^*(\eta_{PX})^*(\sigma_{PX})^*(P\sigma_X)_*(\sigma_{TX})_*
=\\
&=(\eta_{TX})^*\cdot (\sigma_{TX})^*(P\sigma_X)^*(P\sigma_X)_*(\sigma_{TX})_*=(\eta_{TX})^*.\end{align*}
\end{proof}

\section{Presheaf submonads and Beck-Chevalley conditions}

In this section, for a general 2-monad $\TT=(T,\mu,\eta)$ on $\VCat$, we relate its BC* properties with the existence of a (sub)monad morphism $\TT\to\PP$. We remark that a necessary condition for $\TT$ to be a submonad of $\PP$ is that $TX$ is separated for every $V$-category $X$, since $PX$ is separated and separated $V$-categories are stable under monomorphisms.

\begin{theorem}\label{th:submonad}
For a 2-monad $\TT=(T,\mu,\eta)$ on $\VCat$ with $TX$ separated for every $V$-category $X$, the following assertions are equivalent:
\begin{tfae}
\item $\TT$ is a submonad of $\PP$.
\item $\TT$ is lax idempotent and satisfies BC*, and both $\eta_X$ and $Q\eta_X\cdot\yoneda_{TX}$ are fully faithful, for each $V$-category $X$.
\item $\TT$ is lax idempotent, $\mu$ and $\eta$ satisfy BC*, and both $\eta_X$ and $Q\eta_X\cdot\yoneda_{TX}$ are fully faithful, for each $V$-category $X$.
\item $\TT$ is lax idempotent, $\eta$ satisfies BC*, and both $\eta_X$ and $Q\eta_X\cdot\yoneda_{TX}$ are fully faithful, for each $V$-category $X$.
\end{tfae}
\end{theorem}

\begin{proof}
(i) $\Rightarrow$ (ii): By (i) there exists a monad morphism $\sigma\colon \TT\to \PP$ with $\sigma_X$ an embedding for every $V$-category $X$.
By Corollary \ref{cor:laxidpt}, with $\PP$, also $\TT$ is lax idempotent.
Moreover, from $\sigma_X\cdot \eta_X=\yoneda_X$, with $\yoneda_X$, also $\eta_X$ is fully faithful. (In fact this is valid for any monad with a monad morphism into $\PP$.)

To show that $\TT$ satisfies BC* we use the characterisation of Theorem \ref{th:Phi}; that is, we know that there is an admissible class $\Phi$ of distributors so that $\TT=\Phi$. Then BC* for $T$ follows directly from the fact that $\Phi f$ is a (co)restriction of $Pf$, for every $V$-functor $f$.

BC* for $\eta$ follows from BC* for $\yoneda$ and full faithfulness of $\sigma$ since, for any commutative diagram in $\VCat$
\[\xymatrix{\cdot\ar[r]\ar[d]&\cdot\ar[r]^f\ar[d]&\cdot\ar[d]\\
\cdot\ar@{}[ru]|{\fbox{1}}\ar[r]&\ar@{}[ru]|{\fbox{2}}\cdot\ar[r]_g&\cdot}\]
with $\fbox{1}\fbox{2}$ satisfying BC*, and $f$ and $g$ fully faithful, also $\fbox{1}$ satisfies BC*.

Thanks to Proposition \ref{prop:laxidpt}, BC* for $\mu$ follows directly from lax idempotency of $\TT$.\\

The implications (ii) $\Rightarrow$ (iii) $\Rightarrow$ (iv) are obvious.\\

(iv) $\Rightarrow$ (i): For each $V$-category $(X,a)$, we denote by $\ha$ the $V$-category structure on $TX$, and define
the $V$-functor $(\xymatrix{TX\ar[r]^-{\sigma_X}&PX})=(\xymatrix{TX\ar[r]^-{\yoneda_{TX}}&PTX\ar[r]^-{Q\eta_X}&PX})$;
that is, $\sigma_X(\fx)=(\xymatrix{X\ar[r]^{\eta_X}&TX\ar[r]|-{\object@{|}}^{\ha}&TX\ar[r]|-{\object@{|}}^-{\fx^\circ}&E})=\ha(\eta_X(\;),\fx)$. As a composite of fully faithful $V$-functors, $\sigma_X$ is fully faithful; moreover, it is an embedding because, by hypothesis, $TX$ and $PX$ are separated $V$-categories.\\

To show that \emph{$\sigma=(\sigma_X)_X\colon T\to P$ is a natural transformation}, that is, for each $V$-functor $f\colon X\to Y$,
the outer diagram
\[\xymatrix{TX\ar[r]^{\yoneda_{TX}}\ar[d]_{Tf}&PTX\ar[r]^{Q\eta_X}\ar[d]|{PTf}&PX\ar[d]^{Pf}\\
TY\ar@{}[ru]|{\fbox{1}}\ar[r]_{\yoneda_{TY}}&PTY\ar@{}[ru]|{\fbox{2}}\ar[r]_{Q\eta_Y}&PY}\]
commutes, we only need to observe that $\fbox{1}$ is commutative and BC* for $\eta$ implies that $\fbox{2}$ is commutative.\\

It remains to show \emph{$\sigma$ is a monad morphism}: for each $V$-category $(X,a)$ and $x\in X$, \[(\sigma_X\cdot\eta_X)(x)=\ha(\eta_X(\;),\eta_X(x))=a(-,x)=x^*=\yoneda_X(x),\] and so $\sigma\cdot \eta=\yoneda$.
To check that, for every $V$-category $(X,a)$, the following diagram commutes
\[\xymatrix{TTX\ar[r]^{\sigma_{TX}}\ar[d]_\mu&PTX\ar[r]^{P\sigma_X}&PPX\ar[d]^{\mult_X}\\
TX\ar[rr]_{\sigma_X}&&PX,}\]
let $\fX\in TTX$. We have
\begin{align*}
\mult_X\cdot P\sigma_X\cdot \sigma_{TX}(\fX)&=(\xymatrix{X\ar[r]^{\yoneda_X}&PX\ar[r]|-{\object@{|}}^{\ta}&PX\ar[r]|-{\object@{|}}^{\sigma_X^\circ}&
TX\ar[r]^{\eta_{TX}}&TTX\ar[r]|-{\object@{|}}^{\hha}&TTX\ar[r]|-{\object@{|}}^-{\fX^\circ}&E})\\
&=(\xymatrix{X\ar[r]^{\eta_X}&TX\ar[r]|-{\object@{|}}^{\ha}&TX\ar[r]^{\eta_{TX}}&TTX\ar[r]|-{\object@{|}}^{\hha}&TTX\ar[r]|-{\object@{|}}^-{\fX^\circ}&E}),
\end{align*}
since $\sigma_X^\circ\cdot\ta\cdot\yoneda_X(x,\fx)=\ta(\yoneda_X(x),\sigma_X(\fx))=\sigma_X(\fx)(x)=\ha\cdot\eta_X(x,\fx)$, and
\[\sigma_X\cdot \mu_X(\fx)=(\xymatrix{X\ar[r]^{\eta_X}&TX\ar[r]|-{\object@{|}}^{\ha}&TX\ar[r]|-{\object@{|}}^{\mu_X^\circ}&TTX\ar[r]|-{\object@{|}}^-{\fX^\circ}&E}).\]
Hence the commutativity of the diagram follows from the equality $\hha\cdot \eta_{TX}\cdot\ha\cdot\eta_X=\mu_X^\circ\cdot \ha\cdot\eta_X$ we show next. Indeed,
\[\hha\cdot\eta_{TX}\cdot\ha\cdot\eta_X=(\eta_{TX})_*(\eta_X)_*=(\eta_{TX}\cdot\eta_X)_*=(T\eta_X\cdot\eta_X)_*=(T\eta_X)_*(\eta_X)_*=
\mu_X^* (\eta_X)_*=\mu_X^\circ\cdot\ha\cdot\eta_X.\]
\end{proof}

The proof of the theorem allows us to conclude immediately the following result.
\begin{corollary}\label{cor:morphism}
Given a 2-monad $\TT=(T,\mu,\eta)$ on $\VCat$ such that $\eta$ satisfies BC*, there is a monad morphism $\TT\to\PP$ if, and only if, $\eta$ is pointwise fully faithful.
\end{corollary}

\section{On algebras for submonads of $\PP$: a survey}

In the remainder of this paper we will study, given a submonad $\TT$ of $\PP$, the category $(\VCat)^\TT$ of (Eilenberg-Moore) $\TT$-algebras. Here we collect some known results which will be useful in the following sections. We will denote by $\Phi(\TT)$ the admissible class of distributors that induces the monad $\TT$ (defined in \eqref{eq:fai}).

The following result, which is valid for any lax-idempotent monad $\TT$, asserts that, for any $V$-category, to be a $\TT$-algebra is a property (see, for instance, \cite{EF99} and \cite{CLF20}).

\begin{theorem}\label{th:KZ}
Let $\TT$ be lax idempotent monad on $\VCat$.
\begin{enumerate}
\item For a $V$-category $X$, the following assertions are equivalent:
\begin{tfae}
\item[\em (i)] $\alpha\colon TX\to X$ is a $\TT$-algebra structure on $X$;
\item[\em (ii)] there is a $V$-functor $\alpha\colon TX\to X$ such that $\alpha\dashv\eta_X$ with $\alpha\cdot\eta_X=1_X$;
\item[\em (iii)] there is a $V$-functor $\alpha\colon TX\to X$ such that $\alpha\cdot\eta_X=1_X$;
\item[\em (iv)] $\alpha\colon TX\to X$ is a split epimorphism in $\VCat$.
\end{tfae}
\item If $(X,\alpha)$ and $(Y,\beta)$ are $\TT$-algebra structures, then every $V$-functor $f\colon X\to Y$ satisfies $\beta\cdot Tf\leq f\cdot\alpha$.
\end{enumerate}
\end{theorem}

Next we formulate characterisations of $\TT$-algebras that can be found in \cite{Ho11, CH08}, using \emph{injectivity} with respect to certain \emph{embeddings}, and using the existence of certain \emph{weighted colimits}, notions that we recall very briefly in the sequel.

\begin{definition}\cite{Es98}
A $V$-functor $f\colon X\to Y$ is a \emph{$T$-embedding} if $Tf$ is a left adjoint right inverse; that is, there exists a $V$-functor $Tf_\sharp$ such that $Tf\dashv Tf_\sharp$ and $Tf_\sharp\cdot Tf=1_{TX}$.
\end{definition}

For each submonad $\TT$ of $\PP$, the class $\Phi(\TT)$ allows us to identify easily the $T$-embeddings.

\begin{prop}\label{prop:emb}
For a $V$-functor $h\colon X\to Y$, the following assertions are equivalent:
\begin{tfae}
\item $h$ is a $T$-embedding;
\item $h$ is fully faithful and $h_*$ belongs to $\Phi(\TT)$.
\end{tfae}
In particular, $P$-embeddings are exactly the fully faithful $V$-functors.
\end{prop}

\begin{proof}
(ii) $\Rightarrow$ (i): Let $h$ be fully faithful with $h_*\in\Phi(\TT)$. As in the case of the presheaf monad, $\Phi h:\Phi X\to\Phi Y$ has always a right adjoint whenever $h_*\in\Phi(\TT)$, $\Phi^\dashv h:=(-)\cdot h_*\colon \Phi Y\to\Phi X$; that is, for each distributor $\psi:Y\dist E$ in $\Phi Y$, $\Phi^\dashv h(\psi)=\psi\cdot h_*$, which is well defined because by hypothesis $h_*\in\Phi(\TT)$. If $h$ is fully faithful, that is, if $h^*\cdot h_*=(1_X)^*$, then $(\Phi^\dashv h\cdot \Phi h)(\varphi)=\varphi\cdot h^*\cdot h_*=\varphi$.

(i) $\Rightarrow$ (ii): If $\Phi^\dashv h$ is well-defined, then $y^*\cdot h_*$ belongs to $\Phi(\TT)$ for every $y\in Y$, hence $h_*\in \Phi(\TT)$, by \ref{def:admi}(3), and so $h_*\in\Phi(\TT)$. Moreover, if $\Phi^\dashv h\cdot \Phi h=1_{\Phi X}$, then in particular $x^*\cdot h^*\cdot h_*=x^*$, for every $x\in X$, which is easily seen to be equivalent to $h^*\cdot h_*=(1_X)^*$.
\end{proof}

In $\VDist$, given a $V$-distributor $\varphi\colon (X,a)\dist (Y,b)$, the functor $(\;)\cdot\varphi$ preserves suprema, and therefore it has a right adjoint $[\varphi,-]$ (since the hom-sets in $\VDist$ are complete ordered sets):
\[\Dist(X,Z)\adjunct{(\;)\cdot\varphi}{[\varphi,-]}\Dist(Y,Z).\]
For each distributor $\psi\colon X\dist Z$,
\[\xymatrix{X\ar[r]|-{\circ}^{\psi}\ar[d]|-{\circ}_{\varphi}&Z\\
Y\ar@{}[ru]^{\leq}\ar[ru]|-{\circ}_{[\varphi,\psi]}}\]
$[\varphi,\psi]\colon Y\dist Z$ is defined by
\[ [\varphi,\psi](y,z)=\bigwedge_{x\in X}\,\hom(\varphi(x,y),\psi(x,z)).\]

\begin{definitions}\begin{enumerate}
\item Given a $V$-functor $f\colon X\to Z$ and a distributor (here called \emph{weight}) $\varphi\colon X\dist Y$, a \emph{$\varphi$-weighted colimit} of $f$ (or simply a  \emph{$\varphi$-colimit} of $f$), whenever it exists, is a $V$-functor $g\colon Y\to Z$ such that $g_*=[\varphi,f_*]$. One says then that \emph{$g$ represents $[\varphi,f_*]$}.
\item A $V$-category $Z$ is called \emph{$\varphi$-cocomplete} if it has a colimit for each weighted diagram with weight $\varphi\colon(X,a)\dist(Y,b)$; i.e. for each $V$-functor $f\colon X\to Z$, the $\varphi$-colimit of $f$ exists.
\item Given a class $\Phi$ of $V$-distributors, a $V$-category $Z$ is called \emph{$\Phi$-cocomplete} if it is $\varphi$-cocomplete for every $\varphi\in \Phi$. When $\Phi=\VDist$, then $Z$ is said to be \emph{cocomplete}.
\end{enumerate}
\end{definitions}

The proof of the following result can be found in \cite{Ho11, CH08}.
\begin{theorem}\label{th:ch}
Given a submonad $\TT$ of $\PP$, for a $V$-category $X$ the following assertions are equivalent:
\begin{tfae}
\item $X$ is a $\TT$-algebra.
\item $X$ is injective with respect to $T$-embeddings.
\item $X$ is $\Phi(\TT)$-cocomplete.
\end{tfae}
\end{theorem}

$\Phi(\TT)$-cocompleteness of a $V$-category $X$ is guaranteed by the existence of some special weighted colimits, as we explain next. (Here we present very briefly the properties needed. For more information on this topic see \cite{St04}.)

\begin{lemma}
For a distributor $\varphi\colon X\to Y$ and a $V$-functor $f\colon X\to Z$, the following assertions are equivalent:
\begin{tfae}
\item there exists the $\varphi$-colimit of $f$;
\item there exists the $(\varphi\cdot f^*)$-colimit of $1_Z$;
\item for each $y\in Y$, there exists the $(y^*\cdot\varphi)$-colimit of $f$.
\end{tfae}
\end{lemma}
\begin{proof}
(i) $\Leftrightarrow$ (ii): It is straightforward to check that
\[ [\varphi,f_*]=[\varphi\cdot f^*,(1_Z)_*].\]

(i) $\Leftrightarrow$ (iii): Since $[\varphi,f_*]$ is defined pointwise, it is easily checked that, if $g$ represents $[\varphi,f_*]$, then, for each $y\in Y$, the $V$-functor $\xymatrix{E\ar[r]^y&Y\ar[r]^g&Z}$ represents $[y^*\cdot \varphi,f_*]$.

Conversely, if, for each $y\colon E\to Y$, $g_y\colon E\to Z$ represents $[y^*\cdot\varphi,f_*]$, then the map $g\colon Y\to Z$ defined by $g(y)=g_y(*)$ is such that $g_*=[\varphi,f_*]$; hence, as stated in Remark \ref{rem:adjcond}, $g$ is automatically a $V$-functor.
\end{proof}

\begin{corollary}
Given a submonad $\TT$ of $\PP$, a $V$-category $X$ is a $\TT$-algebra if, and only if, $[\varphi, (1_X)_*]$ has a colimit for every $\varphi\in TX$.
\end{corollary}

\begin{remark}
Given $\varphi\colon X\dist E$ in $TX$, in the diagram
\[\xymatrix{X\ar[r]|-{\circ}^{a}\ar[d]|-{\circ}_{\varphi}&X\\
Y\ar@{}[ru]^{\leq}\ar[ru]|-{\circ}_{[\varphi,a]}}\]
\[[\varphi,a](*,x)=\bigwedge_{x'\in X}\hom(\varphi(x',*),a(x',x))=TX(\varphi,x^*).\]
Therefore, if $\alpha\colon TX\to X$ is a $\TT$-algebra structure, then
\[ [\varphi,a](*,x)=TX(\varphi,x^*)=X(\alpha(\varphi),x),\]
that is, $[\varphi,a]=\alpha(\varphi)_*$; this means that $\alpha$ assigns to each distributor $\varphi\colon X\dist E$ the representative of $[\varphi,(1_X)_*]$.
\end{remark}

Hence, we may describe the category of $\TT$-algebras as follows.

\begin{theorem}\label{thm:charact}
\begin{enumerate}
\item A map $\alpha\colon TX\to X$ is a $\TT$-algebra structure if, and only if, for each distributor $\varphi\colon X\dist E$ in $TX$, $\alpha(\varphi)_*=[\varphi,(1_X)_*]$.
\item If $X$ and $Y$ are $\TT$-algebras, then a $V$-functor $f\colon X\to Y$ is a $\TT$-homomorphism if, and only if, $f$ preserves $\varphi$-weighted colimits for any $\varphi\in TX$, i.e., if $x\in X$ represents $[\varphi,(1_X)_*]$, then $f(x)$ represents $[\varphi\cdot f^*,(1_Y)_*]$.
\end{enumerate}
\end{theorem}

\section{On algebras for submonads of $\PP$: the special case of the formal ball monad}

From now on we will study more in detail $(\VCat)^\TT$ for special submonads $\TT$ of $\PP$. In our first example, the formal ball monad $\BB$, we will need to consider the (co)restriction of $\BB$ and $\PP$ to $\VCat_\sep$. We point out that the characterisations of $\TT$-algebras of Theorem \ref{th:ch} remain valid for these (co)restrictions.

The space of formal balls is an important tool in the study of (quasi-)metric spaces. Given a metric space $(X,d)$ its \emph{space of formal balls} is simply the collection of all pairs $(x,r),$ where $x \in X$ and $r \in [0,\infty[$.  This space can itself be equipped with a (quasi-)metric. Moreover this construction can naturally be made into a monad on the category of (quasi-)metric spaces (cf. \cite{GL19, KW11} and references there).

This monad can readily be generalised to $V$-categories, using a $V$-categorical structure in place of the (quasi-)metric.
We will start by considering an extended version of the formal ball monad, the  \emph{extended formal ball monad} $\BBb,$ which we define below.

\begin{definitions} The \emph{extended formal ball monad} $\BBb=(\Bb ,\eta, \mu)$ is given by the following:
\begin{enumerate}
\item[--] a functor $\Bb\colon\VCat\to\VCat$  which maps each $V$-category $X$  to  $\Bb X$ with underlying set $X\times V$ and   \[\Bb X((x,r),(y,s))= \hm{r}{   X(x,y) \otimes s }\]  and every $V$-functor  $f\colon X \to Y$ to the $V$-functor $\Bb f\colon\Bb X\to \Bb Y$ with $\Bb f(x,r)=(f(x),r)$;
\item[--] natural transformations $\eta\colon 1 \to \Bb$ and $\mu\colon \Bb\Bb \to \Bb$ with $\eta_X(x)=(x,k)$ and $\mu_X((x,r),s)=(x,r\otimes s)$, for every $V$-category $X$, $x\in X$, $r,s\in V$.
\end{enumerate}
The \emph{formal ball monad} $\BB$ is the submonad of $\BBb$ obtained when we only consider balls with radius different from $\bot$.
\end{definitions}

\begin{remark}
Note that $\BBb X$ is not separated if $X$ has more than one element (for any $x,y \in X$, $(x,\bot)\simeq (y,\bot)$), while, as shown in \ref{prop:canc}, for $X$ separated, separation of $\BB X$ depends on an extra property of the quantale $V$.
\end{remark}

Using Corollaries \ref{cor:morphism} and \ref{cor:laxidpt}, it is easy to check that

\begin{prop}\label{prop:Bbmonadmorphismff}
There is a pointwise fully faithful monad morphism $\sigma \colon \BBb \to  \PP$. In particular, both $\BBb$ and $\BB$ are lax-idempotent.
\end{prop}
 \begin{proof}
First of all let us check that $\eta$ satisfies BC*, i.e., for any $V$-functor $f\colon X\to Y$,
\[\xymatrix{X\ar[r]|-{\circ}^{(\eta_X)_*}&\Bb X\\
Y\ar@{}[ru]|{\geq}\ar[u]|-{\circ}^{f^*}\ar[r]|-{\circ}_{(\eta_Y)_*}&\Bb Y\ar[u]|-{\circ}_{(\Bb f)^*}}\]
For $y\in Y$, $(x,r)\in\Bb X$,
\begin{align*}
((\Bb f)^*(\eta_Y)_*)(y,(x,r))&=\Bb Y((y,k),(f(x),r))=Y(y,f(x))\otimes r\\
&\leq \bigvee_{z\in X}Y(y,f(z))\otimes X(z,x)\otimes r=\bigvee_{z\in X} Y(y,f(z))\otimes\Bb X((z,k),(x,r))\\
&=((\eta_X)_*f^*)(y,(x,r)).
\end{align*}

Then, by Corollary \ref{cor:morphism}, for each $V$-category $X$, $\sigma_X$ is defined as in the proof of Theorem \ref{th:submonad}, i.e. for each $(x,r)\in\Bb X$, $\sigma_X(x,r)=\Bb X((-,k),(x,r))\colon X\to V$; more precisely, for each $y\in X$, $\sigma_X(x,r)(y)=X(y,x)\otimes r$.

Moreover, $\sigma_X$ is fully faithful: for each $(x,r), (y,s)\in \Bb X$,
\begin{align*}
\Bb X((x,r),(y,s))&=\hom(r,X(x,y)\otimes s)\geq \hom(X(x,x)\otimes r, X(x,y)\otimes s)\\
&\geq \bigwedge_{z\in X}\hom(X(z,x)\otimes r,X(z,y)\otimes s)=PX(\sigma(x,r),\sigma(y,s)).
\end{align*}
\end{proof}

It is clear that $\sigma\colon\BBb\to\PP$ is not pointwise monic; indeed, if $r=\bot$, then $\sigma_X(x,\bot)\colon X\dist E$ is the distributor that is constantly $\bot$, for any $x\in X$. Still it is interesting to identify the $\BBb$-algebras via the existence of special weighted colimits.

\begin{prop}\label{prop:Balg}
For a $V$-category $X$, the following conditions are equivalent:
\begin{tfae}
\item $X$ has a $\BBb$-algebra structure $\alpha\colon\Bb X\to X$;
\item $(\forall x\in X)\;(\forall r\in V)\;(\exists x\oplus r\in X)\;(\forall y\in X)\;\; X(x\oplus r,y)=\hom(r,X(x,y))$;
\item for all $(x,r)\in\Bb X$, every diagram of the sort
\[\xymatrix{X\ar[r]|-{\circ}^{(1_X)_*}\ar[d]|-{\circ}_{\sigma_X(x,r)}&X\\
E\ar@{}[ru]^{\leq}\ar[ru]|-{\circ}_{[\sigma_X(x,r),(1_X)_*]}}\]
has a (weighted) colimit.
\end{tfae}
\end{prop}
\begin{proof}
(i) $\Rightarrow$ (ii): The adjunction $\alpha\dashv\eta_X$ gives, via Remark \ref{rem:adjcond},
\[X(\alpha(x,r),y)=\Bb X((x,r),(y,k))=\hom(r,X(x,y)).\]
For $x\oplus r:=\alpha(x,r)$, condition (ii) follows.\\

(ii) $\Rightarrow$ (iii): The calculus of the distributor $[\sigma_X(x,r),(1_X)_*]$ shows that it is represented by $x\oplus r$:
\[ [\sigma_X(x,r),(1_X)_*](*,y)=\hom(r,X(x,y)).\]

(iii) $\Rightarrow$ (i) For each $(x,r)\in \Bb X$, let $x\oplus r$ represent $[\sigma_X(x,r),(1_X)_*]$. In case $r=k$, we choose $x\oplus k=x$ to represent the corresponding distributor (any $x'\simeq x$ would fit here but $x$ is the right choice for our purpose). Then $\alpha\colon\Bb X\to X$ defined by $\alpha(x,r)=x\oplus r$ is, by construction, left adjoint to $\eta_X$, and $\alpha\cdot\eta_X=1_X$.
\end{proof}
The $V$-categories $X$ satisfying (iii), and therefore satisfying the above (equivalent) conditions, are called \emph{tensored}.
 This notion was originally introduced in the article \cite{BK75} by Borceux and Kelly for general $V$-categories (for our special $V$-categories we suggest to consult \cite{St04}).\\

 Note that, thanks to condition (ii), we get the following characterisation of tensored categories.

 \begin{corollary}\label{cor:oplus}
 A $V$-category $X$ is tensored if, and only if, for every $x\in X$,
 \[X\adjunct{x\oplus -}{X(x,-)}V\]
 is an adjunction in $\VCat$.
 \end{corollary}

We now shift our attention  to the formal ball monad $\BB.$
The characterisation of $\BBb$-algebras given by the Proposition \ref{prop:Balg} may be adapted to obtain a characterisation of $\BB$-algebras. Indeed, the only difference is that a $\BB$-algebra structure $BX\to X$ does not include the existence of $x\oplus\bot$ for $x\in X$, which, when it exists, is the top element with respect to the order in $X$. Moreover, the characterisation of $\BB$-algebras given in \cite[Proposition 3.4]{GL19} can readily be generalised to $\VCat$ as follows.

\begin{prop}
For a $V$-functor $\alpha\colon BX\to X$ the following conditions are equivalent.
\begin{tfae}
\item $\alpha$ is a $\BB$-algebra structure.
\item For every $x\in X$, $r,s\in V\setminus\{\bot\}$, $\alpha(x,k)=x$ and $\alpha(x,r\otimes s)=\alpha(\alpha(x,r),s)$.
\item For every $x\in X$, $r\in V\setminus\{\bot\}$, $\alpha(x,k)=x$ and $X(x,\alpha(x,r))\geq r$.
\item For every $x\in X$, $\alpha(x,k)=x$.
\end{tfae}
\end{prop}

\begin{proof}
By definition of $\BB$-algebra, (i) $\Leftrightarrow$ (ii), while (i) $\Leftrightarrow$ (iv) follows from Theorem \ref{th:KZ}, since $\BB$ is lax-idempotent.
(iii) $\Rightarrow$ (iv) is obvious, and so it remains to prove that, if $\alpha$ is a $\BB$-algebra structure, then $X(x,\alpha(x,r))\geq r$, for $r\neq\bot$. But
\[X(x,\alpha(x,r))\geq r\;\Leftrightarrow\; k\leq\hom(r,X(x,\alpha(x,r))=X(\alpha(x,r),\alpha(x,r)),\]
because $\alpha(x,-)\dashv X(x,-)$ by Corollary \ref{cor:oplus}.
\end{proof}

Since we know that, if $X$ has a $\BB$-algebra structure $\alpha$, then $\alpha(x,r)=x\oplus r$, we may state the conditions above as follows.

\begin{corollary}\label{cor:condition}
If $\xymatrix{BX\ar[r]^{-\oplus-}&X}$ is a $\BB$-algebra structure, then, for $x\in X$, $r,s\in V\setminus\{\bot\}$:
\begin{enumerate}
\item $x\oplus k=x$;
\item $x\oplus(r\otimes s)=(x\oplus r)\oplus s$;
\item $X(x,x\oplus r)\geq r$.
\end{enumerate}
\end{corollary}

\begin{lemma}
Let $X$ and  $Y$ be $V$-categories equipped with $\BB$-algebra structures $\xymatrix{BX\ar[r]^{-\oplus-}&X}$ and $\xymatrix{BY\ar[r]^{-\oplus-}&Y}$. Then a map $f: X \rightarrow Y$ is a $V$-functor if and only if $$ f \textrm{ is monotone and }  f(x) \oplus r \leq f(x \oplus r) ,$$ for all $(x,r) \in BX$.
\end{lemma}

\begin{proof}
Assume that $f$ is a $V$-functor. Then it is, in particular, monotone, and, from Theorem \ref{th:KZ} we know that $f(x)\oplus r\leq f(x\oplus r)$.

Conversely, assume that $f$ is monotone and that $f(x) \oplus r \leq f(x \oplus r),$ for all $(x,r) \in BX $.
Let $x,x' \in X$. Then $x\oplus X(x,x')\leq x'$ since $(x\oplus -)\dashv X(x,-)$ by Corollary \ref{cor:oplus}, and then
\begin{align*}
f(x)\oplus X(x,x')&\leq f(x\oplus X(x,x'))&\mbox{(by hypothesis)}\\
&\leq f(x')&\mbox{(by monotonicity of $f$).}
\end{align*}
Now, using the adjunction $ f(x)\oplus - \dashv Y(f(x),-) )$, we conclude that \[X(x,x') \leq Y(f(x),f(x')).\] \end{proof}

The following results are now immediate:

\begin{corollary}
\begin{enumerate}
\item Let $(X,\oplus), (Y,\oplus)$ be $\BB$-algebras. Then a map $f\colon X \rightarrow Y$ is a $\BB$-algebra morphism if and only if, for all $(x,r) \in BX$, \[f \textrm{ is monotone and }  f(x \oplus r)= f(x) \oplus r.\]
\item Let $(X,\oplus), (Y,\oplus)$ be $\BB$-algebras. Then a $V$-functor $f\colon X \rightarrow Y$ is a $\BB$-algebra morphism if and only if, for all $(x,r) \in BX$, \[f(x \oplus r)\leq f(x) \oplus r.\]
\end{enumerate}
\end{corollary}

\begin{example}
If $X\subseteq\,[0,\infty]$, with the $V$-category structure inherited from $\hom$, then
\begin{enumerate}
\item $X$ is a $\BBb$-algebra if, and only if, $X=[a,b]$ for some $a,b\in\,[0,\infty]$.
\item $X$ is a $\BB$-algebra if, and only if, $X=\,]a,b]$ or $X=[a,b]$ for some $a,b\in\,[0,\infty]$.
\end{enumerate}
Let $X$ be a $\BBb$-algebra. From Proposition \ref{prop:Balg} one has
\[(\forall x\in X)\;(\forall r\in\,[0,\infty])\;(\exists x\oplus r\in X)\;(\forall y\in X)\;\;y\ominus (x\oplus r)=(y\ominus x)\ominus r=y\ominus (x+r).\]
This implies that, if $y\in X$, then $y>x\otimes r\;\Leftrightarrow\;y>x+r$. Therefore, if $x+r\in X$, then $x\oplus r=x+r$, and, moreover, $X$ is an interval: given $x,y,z\in\,[0,\infty]$ with $x<y<z$ and $x,z\in X$, then, with $r=y-x\in\,[0,\infty]$, $x+r=y$ must belong to $X$:
\[z\ominus(x\oplus r)=z-(x+r)=z-y>0\;\Rightarrow\;z\ominus(x\oplus r)=z-(x\oplus r)=z-y\;\Leftrightarrow\; y=x\oplus r\in X.\]
In addition, $X$ must have bottom element (that is a maximum with respect to the classical order of the real half-line): for any $x\in X$ and $b=\sup X$, $x\oplus(b-x)=\sup\{z\in X\,;\,z\leq b\}=b\in X$. For $r=\infty$ and any $x\in X$, $x\oplus\infty$ must be the top element of $X$, so $X=[a,b]$ for $a,b\in\,[0,\infty]$.

Conversely, if $X=]a,b]$, for $x\in X$ and $r\in\,[0,\infty[$, define $x\oplus r=x+r$ if $x+r\in X$ and $x\oplus r=b$ elsewhere. It is easy to check that condition (ii) of Proposition \ref{prop:Balg} is satisfied for $r\neq\infty$.

Analogously, if $X=[a,b]$, for $x\in X$ and $r\in\,[0,\infty]$, we define $x\oplus r$ as before in case $r\neq\infty$ and $x\oplus\infty=a$.
\end{example}

As we will see, (co)restricting $\BB$ to $\VCat_\sep$ will allows us to obtain some interesting results.   Unfortunately  $X$ being separated does not entail $BX$ being so. Because of this we will need to restrict our attention to the \textit{cancellative} quantales which we define and characterize next.

\begin{definition}
A quantale $V$ is said to be \emph{cancellative} if
\begin{equation}\label{eq:canc}
\forall r,s \in V,\, r\neq \bot :\  r=s \otimes r  \ \Rightarrow \  s=k.
\end{equation}
\end{definition}
\begin{remark}
We point out that this notion of cancellative quantale does not coincide with the notion of cancellable ccd quantale introduced in \cite{CH17}. On the one hand cancellative quantales are quite special, since, for instance, when $V$ is a locale, and so with $\otimes=\wedge$ is a quantale, $V$ is not cancellative since condition \eqref{eq:canc} would mean, for $r\neq\bot$, $r=s\wedge r\;\Rightarrow\;s=\top$. On the other hand, $[0,1]_\odot$, that is $[0,1]$ with the usual order and having as tensor product the \L{}ukasiewicz sum, is cancellative but not cancellable.
In addition we remark that every \emph{value quantale} \cite{KW11} is cancellative.
\end{remark}

\begin{prop}\label{prop:canc}
Let $V$ be an integral quantale. The following assertions are equivalent:
\begin{tfae}
\item $BV$ is separated;
\item $V$ is cancellative;
\item If $X$ is separated then $BX$ is separated.
\end{tfae}
\end{prop}
\begin{proof}
(i) $\Rightarrow$ (ii): Let $ r,s \in V,\, r\neq \bot $ and $  r=s \otimes r$. Note that \[ BV((k,r),(s,r))=\hm{r}{\hm{k}{s}\otimes r}=\hm{r}{s \otimes r}=\hm{r}{r}=k\] and \[BV((s,r),(k,r))=\hm{r}{\hm{s}{k}\otimes r}=\hm{r}{\hm{s}{k}\otimes s \otimes  r} =\hm{s \otimes r }{  s \otimes  r}=k.\] Therefore, since $BV$ is separated, $(s,r)=(k,r)$ and it follows that $s=k.$\\

(ii) $\Rightarrow $ (iii): If $(x,r)\simeq (y,s)$ in $BX$, then
\[BX((x,r),(y,s))=k \Leftrightarrow r \leq X(x,y) \otimes s, \mbox{ and }\]
\[BX((y,s),(x,r))=k \Leftrightarrow s \leq X(y,x) \otimes r.\]
Therefore $r\leq s$ and $s \leq r$, that is $r=s.$   Moreover, since $r \leq X(x,y) \otimes r \leq r$ we have that $X(x,y)=k$. Analogously, $X(y,x)=k$ and we conclude that $x=y$.\\

(iii) $\Rightarrow$ (i): Since $V$ is separated it follows immediately from (iii) that $BV$ is separated.
  \end{proof}

We can now show that $\BB$ is a submonad of $\PP$ in the adequate setting. \emph{From now on we will be working with a cancellative and integral quantale $V$, and $\BB$ will be the (co)restriction of the formal ball monad to $\VCats$.}
\begin{prop}
 Let $V$ be a cancellative and integral quantale. Then $\BB$ is a submonad of $\PP$ in $\VCats$.
\end{prop}
\begin{proof}
Thanks to Proposition \ref{prop:Bbmonadmorphismff}, all that remains is to show  that $\sigma_X $ is injective on objects, for any $V$-category $X$. Let $\sigma(x,r)=\sigma(y,s)$, or, equivalently, $X(-,x)\otimes r =X(-,y)\otimes s$. Then, in particular, \[r = X(x,x)\otimes r = X(x,y) \otimes s \leq s= X(y,y)\otimes s = X(y,x)\otimes r \leq r.\]
 Therefore $r=s$ and $X(y,x)=X(x,y)=k$. We conclude that $(x,r)=(y,s)$.
\end{proof}

Thanks to Theorem \ref{th:ch} $\BB$-algebras are characterized via an injectivity property with respect to special embeddings. We end this section studying in more detail these embeddings.
Since we are working in $\VCats$, a $B$-embedding $h\colon X\to Y$, being fully faithful, is injective on objects. Therefore, for simplicity, we may think of it as an inclusion. With $Bh_\sharp\colon BY\to BX$ the right adjoint and left inverse of $Bh\colon BX\to BY$, we denote $Bh_\sharp(y,r)$ by $(y_r, r_y)$.

\begin{lemma}\label{prop:h}
Let $h\colon X\to Y$ be a $B$-embedding. Then:
\begin{enumerate}
\item $(\forall y\in Y)\;(\forall x\in X)\;(\forall r\in V)\; BY((x,r),(y,r))=BY((x,r),(y_r,r_y))$;
\item $(\forall \, y \in Y) \colon  k_y=Y(y_k,y)$;
\item $(\forall\, y \in Y)\;(\forall x \in X)\colon \enskip Y(x,y)= Y(x,y_k)\otimes Y(y_k,y)$.
\end{enumerate}
\end{lemma}

\begin{proof}
(1) From $Bh_\sharp\cdot Bh=1_{BX}$ and $Bh\cdot Bh_\sharp\leq 1_{BY}$ one gets, for any $(y,r)\in BY$, $(y,r)\leq (y_r,r_y)$, i.e. $BY((y,r),(y_r,r_y))=\hom(r_y,Y(y_r,y)\otimes r)=k$. Therefore,  for all $x\in X$, $y\in Y$, $r\in V$,
\begin{align*}
BY((x,r),(y,r))&\leq BX((x,r),(y_r,r_y))=BY((x,r),(y_r,r_y))\\
&=BY((x,r),(y_r,r_y))\otimes BY((y_r,r_y),(y,r))\leq BY((x,r),(y,r)),
\end{align*}
that is
\[BY((x,r),(y,r))=BY((x,r),(y_r,r_y)).\]

(2) Let $y \in Y$. Then
\[Y(y_k,y)=BY((y_k,k),(y,k))=BY((y_k,k),(y_k,k_y))=k_y.\]

(3) Let $y\in Y$ and $x\in X$. Then
\[Y(x,y)=BY((x,k),(y,k))=BY((x,k),(y_k,k_y))=Y(x,y_k)\otimes k_y=Y(x,y_k)\otimes Y(y_k,y).\]

\end{proof}

\begin{prop}
Let $X$ and $Y$ be $V$-categories. A $V$-functor $h\colon X\to Y$ is a $B$-embedding if and only if $h$ is fully faithful and
\begin{equation}\label{eq:fff}
(\forall y \in Y)\;(\exists ! z\in X)\; (\forall x\in X)\;\;\; Y(x,y)=Y(x,z)\otimes Y(z,y).
\end{equation}
\end{prop}

\begin{proof}
If $h$ is a $B$-embedding, then it is fully faithful by Proposition \ref{prop:emb} and, for each $y\in Y$, $z=y_k\in X$ fulfils the required condition.
To show that such $z$ is unique, assume that $z,z'\in X$ verify the equality of condition \eqref{eq:fff}. Then
\[Y(z,y)=Y(z,z')\otimes Y(z',y)\leq Y(z',y)=Y(z',z)\otimes Y(z,y)\leq Y(z,y),\]
and therefore, because $V$ is cancellative, $Y(z',z)=k$; analogously one proves that $Y(z,z')=k$, and so $z=z'$ because $Y$ is separated.\\

To prove the converse, for each $y\in Y$ we denote by $\yk$ the only $z\in X$ satisfying \eqref{eq:fff}, and define \[Bh_\sharp(y,r)=(\yk,Y(\yk,y)\otimes r).\] When $x\in X$, it is immediate that $\overline{x}=x$, and so $Bh_\sharp\cdot Bh=1_{BX}$. Using Remark \ref{rem:adjcond}, to prove that $Bh_\sharp$ is a $V$-functor and $Bh\dashv Bh_\sharp$ it is enough to show that
\[BX((x,r),Bh_\sharp(y,s))=BY(Bh(x,r),(y,s)),\]
for every $x\in X$, $y\in Y$, $r,s\in V$. By definition of $Bh_\sharp$ this means
\[BX((x,r),(\yk,Y(\yk,y)\otimes s))=BY((x,r),(y,s)),\]
that is,
\[\hom(r,Y(x,\yk)\otimes Y(\yk,y)\otimes s)=\hom(r,Y(x,y)\otimes s),\]
which follows directly from \eqref{eq:fff}.
\end{proof}

\begin{corollary}
In $\Met$, if $X\subseteq [0,\infty]$, then its inclusion $h\colon X\to[0,\infty]$ is a $B$-embedding if, and only if, $X$ is a closed interval.
\end{corollary}
\begin{proof}
If $X=[x_0,x_1]$, with $x_0,x_1\in\,[0,\infty]$, $x_0\leq x_1$, then it is easy to check that, defining $\yk=x_0$ if $y\leq x_0$, $\yk=y$ if $y\in X$, and $\yk=x_1$ if $y\geq x_1$, for every $y\in\,[0,\infty]$, condition \eqref{eq:fff} is fulfilled.\\

We divide the proof of the converse in two cases:

(1) If $X$ is not an interval, i.e. if there exists $x,x'\in X$, $y\in [0,\infty]\setminus X$  with $x<y<x'$, then either $\yk<y$, and then
\[0=y\ominus x'\neq (y\ominus x')+(y\ominus\yk)=y-\yk,\]
or $\yk>y$, and then
\[y-x=y\ominus x\neq (\yk\ominus x)+(y\ominus\yk)=\yk-x.\]\\

(2) If $X=[x_0,x_1[$ and $y> x_1$, then there exists $x\in X$ with $\yk<x<y$, and so
\[y-x=y\ominus x\neq (\yk\ominus x)+(y\ominus\yk)=y-\yk.\]
An analogous argument works for $X=]x_0,x_1]$.
\end{proof}

\section{On algebras for submonads of $\PP$ and their morphisms}

In the following $\TT=(T,\mu,\eta)$ is a submonad of the presheaf monad $ \PP=(P,\mult,\yoneda)$ in $\VCats$ For simplicity we will assume that the injective and fully faithful components of the monad morphism $\sigma:T \rightarrow P$ are inclusions. Theorem \ref{th:KZ} gives immediately that:

\begin{prop}
Let  $(X,a)$ be a $V$-category and $\alpha: TX \rightarrow X$   be a  $V$-functor.  The following are equivalent: \begin{enumerate}
 \item $(X,\alpha)$ is a $\TT$-algebra;
  \item $\forall\, x \in X:$ $ \alpha  (x^*)=x $.
 \end{enumerate}
\end{prop}

We would like to identify the $\TT$-algebras directly, as we did for $\BBb$ or $\BB$ in Proposition \ref{prop:Balg}. First of all, we point out that a $\TT$-algebra structure $\alpha\colon TX\to X$ must satisfy, for every $\varphi\in TX$ and $x\in X$,
\[X(\alpha(\varphi), x)=TX(\varphi,x^*),\]
and so, in particular,
\[\alpha(\varphi)\leq x \;\Leftrightarrow\;\varphi\leq x^*;\]
hence $\alpha$ must assign to each $\varphi\in TX$ an $x_\varphi\in X$ so that
\[x_\varphi=\min\{x\in X\,;\,\varphi\leq x^*\}.\] Moreover, for such map $\alpha\colon TX\to X$, $\alpha$ is a $V$-functor if, and only if,
\begin{align*}
&\;(\forall \varphi,\rho\in TX)\;\;TX(\varphi,\rho)\leq X(x_\varphi,x_\rho)=TX(X(-,x_\varphi),X(-,x_\rho))\\
\Leftrightarrow&\;(\forall \varphi,\rho\in TX)\;\;TX(\varphi,\rho)\leq \bigwedge_{x\in X}\hom(X(x,x_\varphi),X(x,x_\rho))\\
\Leftrightarrow&\;(\forall x\in X)\;(\forall \varphi,\rho\in TX)\;\;X(x,x_\varphi)\otimes TX(\varphi,\rho)\leq X(x,x_\rho).
\end{align*}

\begin{prop}
A $V$-category $X$ is a $\TT$-algebra if, and only if:
\begin{enumerate}
\item for all $\varphi\in TX$ there exists $\min\{x\in X\,;\,\varphi\leq x^*\}$;
\item for all $\varphi, \rho\in TX$ and for all $x\in X$, $X(x,x_\varphi)\otimes TX(\varphi,\rho)\leq X(x,x_\rho)$.
\end{enumerate}
\end{prop}
We remark that condition (2) can be equivalently stated as:
\begin{enumerate}
\item[\emph{(2')}] for each $\rho\in TX$, the distributor $\rho_1=\displaystyle\bigvee_{\varphi\in TX} X(-,x_\varphi)\otimes TX(\varphi,\rho)$ satisfies $x_{\rho_1}=x_\rho$,
\end{enumerate}
which is the condition corresponding to condition (2) of Corollary \ref{cor:condition}.\\

Finally, as for the formal ball monad, Theorem \ref{th:KZ} gives the following characterisation of $\TT$-algebra morphisms.

\begin{corollary}
Let $(X,\alpha), (Y,\beta)$ be $\TT$-algebras. Then a $V$-functor $f: X \rightarrow Y$ is a $\TT$-algebra morphism if and only if   \[(\forall \varphi \in TX)\;\;\beta(\varphi \cdot f^*) \geq f(\alpha(\varphi)).\]
\end{corollary}

\begin{example}
\textbf{The Lawvere monad.} Among the examples presented in \cite{CH08} there is a special submonad of $\PP$ which is inspired by the crucial remark of Lawvere in \cite{Law73} that Cauchy completeness for metric spaces is a kind of cocompleteness for $V$-categories. Indeed, the submonad $\LL$ of $\PP$ induced by
\[\Phi=\{\varphi\colon X\dist Y\,;\,\varphi\mbox{ is a right adjoint $V$-distributor}\}\]
has as $\LL$-algebras the \emph{Lawvere complete $V$-categories}. These were studied also in \cite{CH09}, and in \cite{HT10} under the name $L$-complete $V$-categories. When $V=[0,\infty]_+$, using the usual order in $[0,\infty]$, for distributors $\varphi\colon X\dist E$, $\psi\colon E\dist X$ to be adjoint
\[\xymatrix@=8ex{X\ar@{}[r]|{\top}\ar@<1mm>@/^2mm/[r]|\circ^{{\varphi}} & \ar@<1mm>@/^2mm/[l]|\circ^{{\psi}}E}\]means that
\begin{align*}
(\forall x,x'\in X)\;\;&X(x,x')\leq \varphi(x)+\psi(x'),\\
&0\geq \inf_{x\in X} (\psi(x)+\varphi(x)).
\end{align*}
This means in particular that
\[(\forall n\in\NN)\;(\exists x_n\in X)\;\;\psi(x_n)+\varphi(x_n)\leq\frac{1}{n},\]
and, moreover,
\[X(x_n,x_m)\leq\varphi(x_n)+\psi(x_m)\leq \frac{1}{n}+\frac{1}{m}.\]
This defines a \emph{Cauchy sequence} $(x_n)_n$, so that
\[(\forall\varepsilon>0)\;(\exists p\in\NN)\;(\forall n,m\in\NN)\;n\geq p\;\wedge\;m\geq p\;\Rightarrow\;\;X(x_n,x_m)+X(x_m,x_n)<\varepsilon.\]
Hence, any such pair induces a (equivalence class of) Cauchy sequence(s) $(x_n)_n$, and a representative for
\[\xymatrix{X\ar[r]|-{\circ}^{(1_X)_*}\ar[d]|-{\circ}_{\varphi}&X\\
E\ar@{}[ru]^{\leq}\ar[ru]|-{\circ}_{[\varphi,(1_X)_*]}}\] is nothing but a limit point for $(x_n)_n$. Conversely, it is easily checked that every Cauchy sequence $(x_n)_n$ in $X$
gives rise to a pair of adjoint distributors
\[\varphi=\lim_n\,X(-,x_n)\mbox{ and }\psi=\lim_n\,X(x_n,-).\]
We point out that the $\LL$-embeddings, i.e. the fully faithful and fully dense $V$-functors $f\colon X\to Y$ do not coincide with the $\LL$-dense ones (so that $f_*$ is a right adjoint). For instance, assuming for simplicity that $V$ is integral, a $V$-functor $y\colon E\to X$ ($y\in X$) is fully dense if and only if $y\simeq x$ for all $x\in X$, while it is an $\LL$-embedding if and only if $y\leq x$ for all $x\in X$. Indeed, $y\colon E\to X$ is $\LL$-dense if, and only if,
\begin{enumerate}
\item[--] there is a distributor $\varphi\colon X\dist E$, i.e.
\begin{equation}\label{eq:distr}
(\forall x,x'\in X)\;\;X(x,x')\otimes\varphi(x')\leq\varphi(x),
\end{equation}
such that
\item[--] $k\geq \varphi\cdot y_*$ , which is trivially true, and $a\leq y_*\cdot\varphi$, i.e.
\begin{equation}\label{eq:adjoint}
(\forall x,x'\in X)\;\;X(x,x')\leq \varphi(x)\otimes X(y,x').
\end{equation}
\end{enumerate}
Since \eqref{eq:distr} follows from \eqref{eq:adjoint},
\[y\mbox{ is $\LL$-dense }\;\Leftrightarrow\;\;(\forall x,x'\in X)\;\;X(x,x')\leq \varphi(x)\otimes X(y,x').\]
In particular, when $x=x'$, this gives $k\leq \varphi(x)\otimes X(y,x)$, and so we can conclude that, for all $x\in X$, $y\leq x$ and $\varphi(x)=k$. The converse is also true; that is
\[y\mbox{ is $\LL$-dense }\;\Leftrightarrow\;\;(\forall x\in X)\;\;y\leq x.\]

Still, it was shown in \cite{HT10} that injectivity with respect to fully dense and fully faithful $V$-functors (called $L$-dense in \cite{HT10}) characterizes also the $\LL$-algebras.

\end{example}

\section*{Acknowledgements}

We are grateful to Dirk Hofmann for useful discussions concerning our last example.\\

This work was partially supported by the Centre for Mathematics of the University of Coimbra -- UIDB/00324/2020, funded by the Portuguese Government through FCT/MCTES, and the FCT PhD grant SFRH/BD/150460/2019.

\end{document}